\documentclass[]{article}
\usepackage{amssymb,amsmath,amsthm,mathdesign}
\usepackage{pdfsync}
\usepackage{color}
	
\newcommand{\Z}{{\mathbf Z}}
\newcommand{\R}{\mathbf{R}}
\newcommand{\N}{\mathbf{N}}

\newcommand {\E}{\mathrm{E}}

\renewcommand{\d}{\text{\rm d}}

\newcommand{\sG}{\mathcal{G}}

\newcommand{\sA}{\mathcal{A}}
\newcommand{\sE}{\mathcal{E}}

\newtheorem{stat}{Statement}[section]
\newtheorem{proposition}[stat]{Proposition}

\newtheorem{theorem}[stat]{Theorem}
\newtheorem{lemma}[stat]{Lemma}

\theoremstyle{definition}
\newtheorem{remark}[stat]{Remark}

\numberwithin{equation}{section}

\begin{document}

\title{\bf Moment bounds for a class of fractional stochastic heat equations%
	\thanks{%
	Research supported in part by EPSRC.}}
	
\author{Mohammud Foondun\\Loughborough University
\and Wei Liu \\Loughborough University
\and McSylvester Omaba\\ Loughborough University }

\date{}
\maketitle
\begin{abstract}
We consider fractional stochastic heat equations of the form $\frac{\partial u_t(x)}{\partial t} = -(-\Delta)^{\alpha/2} u_t(x)+\lambda \sigma (u_t(x)) \dot F(t,\, x)$. Here $\dot F$ denotes the noise term. Under suitable assumptions, we show that the second moment of the solution grows exponentially with time. In particular, this answers an open problem in \cite{CoKh}. Along the way, we prove a number of other interesting properties which extend and complement results in \cite{foonjose}, \cite{Khoshnevisan:2013aa} and \cite{Khoshnevisan:2013ab}.	
		
		
	\noindent{\it Running Title:} Noise excitability and
		parabolic SPDEs.\newpage
\end{abstract}

\section{Introduction and main results.}
Let us look at the following equation, 
\begin{equation*}
\frac{\partial u_t(x)}{\partial t} = -(-\Delta)^{\alpha/2} u_t(x)+\lambda \sigma (u_t(x)) \dot w(t,\, x), \quad \text{for}\quad x \in \R \quad \text{and}\quad t>0,
\end{equation*}
with initial condition $u_0(x)$. The operator $-(-\Delta)^{\alpha/2}$ is the fractional Laplacian of order $1<\alpha \leq 2$.  $\lambda$ is a positive parameter called {\it level of noise} and $\dot{w}$ denotes space-time white noise. The function $\sigma: \R \mapsto \R$ is a Lipschitz function satisfying some growth condition which will be described later.  Since \cite{FK}, it is known that as time $t$ goes to infinity, the second moment of the solution $\E|u_t(x)|^2$ grows like $\exp{(\text{constant} \times t)}$ whenever the initial condition $u_0(x)$ is bounded below.  However, proving the exponential growth when $u_0(x)$ is not bounded below has been a hard open problem even though in \cite{CoKh}, this question has been settled for a different class of equations.  

One of the main aims of this paper is to show that the second moment grows exponentially even if the initial function is not bounded below. This answers an open question in \cite{CoKh}. In fact, we will do much more.  Instead of looking at the equation described above, we will look at the following
\begin{equation}\label{main-eq}
\frac{\partial u_t(x)}{\partial t} = -(-\Delta)^{\alpha/2} u_t(x) + \lambda \sigma (u_t(x)) \dot F(t,\, x) \quad \text{for} ~ x \in \R^d ~\text{and} ~ t > 0,
\end{equation}
where the function $\sigma$ further satisfies $l_\sigma |x|\leq\sigma(x)\leq L_\sigma|x|$ with $l_\sigma$ and $L_\sigma$ being positive constants. $\dot F$ denotes the Gaussian coloured noise satisfying the following property,
\begin{equation*}
\E [\dot F(t, \, x) \dot F(s, \, y)] = \delta_0(t - s) f(x, \, y),
\end{equation*}
where $f$ is the Riesz kernel with parameter $\beta < d$,
\begin{equation*}
f(x, \, y) := \frac{1}{|x - y|^{\beta}}.
\end{equation*}
The initial function $u_0$ is assumed to be a bounded nonnegative function such that there exists a set $A\subset \R^d$ for which \begin{equation*}
\int_{A}u_0(x)\,\d x>0. 
\end{equation*}
 Following Walsh \cite{Walsh}, we define the {\it mild solution} of \eqref{main-eq} as the predictable solution to the following integral equation,
\begin{equation}\label{mild-main}
u_t(x)=
(\sG u)_t(x)+ \lambda \int_{\R^d} \int_0^t p_{t-s}(x,\,y) \sigma(u_s(y))F(\d s\,\d y).
\end{equation}
where
\begin{equation*}
(\sG u)_t(x):=\int_{\R^d} p_t (\,x, \,y) u_0 (y) d y,
\end{equation*}
and $p_t (x, \, y)$ denotes the fractional heat kernel. We will be interested in {\it random field} solutions which require that the mild solution satisfies the following integrability condition 
\begin{equation*}
\sup_{x\in \R^d}\sup_{t>0}\E|u_t(x)|^2<\infty.
\end{equation*}
This will further impose that $\beta\leq \alpha$; see for instance \cite{ferrante}. Existence and uniqueness are well known for the equations being studied here. See \cite{FK} and the references therein.  Here is our first main result for (\ref{main-eq}).
\begin{theorem}\label{colourednoisetheorem}
There exist constants $c$ and $c'$ such that 
\begin{equation*}
\sup_{x \in \R^d} \E |u_t(x)|^2 \leq c \exp \left( c' \lambda^{2 \alpha/ (\alpha - \beta)} t \right)\quad\text{for\,\,all}\quad t>0.
\end{equation*}
And there exists $T > 0$ such that for any $t > T$, 
\begin{equation*}
\inf_{x \in B(0, t^{1/ \alpha})} \E |u_t(x)|^2 \geq \tilde{c} \exp \left( \tilde{c}'\lambda^{2 \alpha/(\alpha - \beta)}t \right), 
\end{equation*}
where $\tilde{c}$ and $\tilde{c}'$ are some positive constants.  This immediately implies that for any fixed $x \in \R^d$,
\begin{equation*}
\tilde{c}' \lambda^{2 \alpha/(\alpha - \beta)} \leq \liminf_{t \rightarrow \infty} \frac{\log \E |u_t(x)|^2 }{t} \leq \limsup_{t \rightarrow \infty} \frac{\log \E |u_t(x)|^2 }{t}\leq c' \lambda^{2 \alpha/(\alpha - \beta)}.
\end{equation*}

\end{theorem}
The first part of this theorem says that second moment grows at most exponentially. While this has been known,  the novelty here is that we give a precise rate with respect to the parameter $\lambda$.  The lower bound is completely new. Most of the results of these kinds have been derived from the renewal theoretic ideas developed in \cite{FK} and \cite{FK2}. The methods used in this article are completely different. In particular, we make use of a localisation argument together with heat kernel estimates for the fractional Laplacian. This allows us to get results for finite times as well.  It should be pointed out that since we are dealing with coloured noise here, the proof of exponential growth of the second moment can be quite involved, even if the initial condition is bounded below. In \cite{FK2}, Fourier analytic methods were used. Here, we develop a much more direct approach which involves a renewal-type inequality for a certain quantity. This greatly simplifies the arguments used in \cite{FK2}. See Remark \ref{correlated} below for a brief discussion.

We can also obtain corresponding bounds for higher moments. We do not pursue this here, because right now we do not know how to get any sharp results for higher moments. See \cite{Davar} for details regarding higher moments.

We also note that the quantity $\frac{2\alpha}{\alpha-\beta}$ is important. It gives information on how the rate of growth depends on the operator and the noise term as well. It will become clear that capturing this dependence has motivated our analysis.

The next theorem gives the rate of growth of the second moment with respect to the parameter $\lambda$, which extends results in \cite{foonjose} and \cite{Khoshnevisan:2013ab}.  We note that for time $t$ large enough, this follows from the theorem above. But for small $t$, we need to work a bit harder. 

\begin{theorem}\label{colourednoiselambdacoro}
For any fixed $t>0$ and $x\in \R^d$, we have
\begin{equation*}
\lim_{\lambda \rightarrow \infty}  \frac{ \log \log \E |u_t(x)|^2}{\log \lambda} = \frac{2 \alpha}{\alpha - \beta}.
\end{equation*}
\end{theorem}

A consequence of the proof of the above theorem is that we can complement some results in \cite{Khoshnevisan:2013ab}, where the authors studied the growth rate of the energy of the solution with respect to $\lambda$.  We follow their notations and define the energy $\sE_t(\lambda)$ of the solution $u_t$ as follows,

\begin{equation}
\sE_t(\lambda):=\sqrt{\int_{\R^d}\E|u_t(x)|^2\,\d x}.
\end{equation}

The above quantity does not always exist. But under suitable assumptions on the initial condition, it does. A bounded non-negative initial condition which is compactly supported is one such condition.

The {\it excitation index} of the solution $u_t$ is defined as follows,
\begin{equation*}
e(t):=\lim_{\lambda \rightarrow \infty}\frac{\log \log \sE_t(\lambda)}{\log \lambda}.
\end{equation*}
We then have the following result.
\begin{theorem}\label{excitation-index}
The excitation index $e(t)$ of the solution to \eqref{main-eq}, if it exists, is $2\alpha/(\alpha-\beta)$.
\end{theorem}

We now give a relationship between the excitation index of \eqref{main-eq} and the continuity property of the solution.

\begin{theorem}\label{continuity}
Let $\eta<(\alpha-\beta)/2\alpha$ then for every $x\in \R^d$, $\{u_t(x),\,t>0\}$, the solution to \eqref{main-eq} has H\"older continuous trajectories with exponent $\eta$. 
\end{theorem}

As such this result can be read off from \cite{BEM2010a} but for the sake of completeness we will give a quick proof.  We include this theorem to make the point that $\eta\leq \frac{1}{e(t)}$, hence showcasing a link between noise excitability and continuity of the solution.  We also have deliberately avoided the point $t=0$ in the statement of the above result. Including this would require some continuity assumption on the initial data in order for the conclusion to remain valid.
The choice of Riesz kernel as correlation function for our noise has been motivated by our desire to calculate the noise excitability.  It will be clear to the reader that exponential growth of the second moment can be proved for a much larger class of correlation functions.

All of our results hold for white noise driven equations as well. We offer the following theorem which emphasises the fact that we have indeed answered the question posed in \cite{CoKh}. We do not keep track of the constants appearing in the proofs of the above results, so we will give a self contained proof of the theorem below. 

\begin{theorem}\label{theoremwhite}
Let $u_t$ denote the unique solution to the following stochastic heat equation,
\begin{equation*}
\frac{\partial u_t(x)}{\partial t} = -(-\Delta)^{\alpha/2} u_t(x) + \lambda \sigma (u_t(x)) \dot w(t,x) \quad \text{for} ~ x \in \R ~\text{and} ~ t > 0, 
\end{equation*}
with conditions described in the introduction. Then, there exists a $T>0$ such that for $t>T$, we have 
\begin{equation*}
\inf_{x \in B(0, t^{1/ \alpha})} \E |u_t(x)|^2 \geq c \exp \left( c' \lambda^{2 \alpha/(\alpha - 1)}t \right), 
\end{equation*}
where $c$ and $c'$ are some positive constants. This immediately yields
\begin{equation*}
\liminf_{t\rightarrow \infty}\frac{1}{t}\log \E |u_t(x)|^2 \geq c' \lambda^{2 \alpha/(\alpha - 1)},
\end{equation*}
for any fixed $x\in \R$.
\end{theorem}
While this paper was in the final stage of preparation, we were informed by Le Chen of \cite{Chen-Dalang}, where the authors have also answered the open problem of \cite{CoKh} mentioned above. In fact, they consider a class of equations which is a bit wider than the one mentioned in the above theorem in that the operator involved is more general. The methods which they employed is also very different and involve some kind of formula for the moments of the solution. See \cite{Chen-Dalang} for more details.  Here, our method is softer and can be applied to equations involving colored noise as well.

We now describe a fundamental strategy upon which our methods rely. We restrict to the situation described in the above theorem. We know from Walsh isometry that the second moment of the solution satisfies 
\begin{equation*}
\E|u_t(x)|^2=(\sG u)_t(x)^2+\lambda^2\int_0^t\int_{\R}p^2_{t-s}(x-y)\E|\sigma(u_s(y))|^2\,\d y\,\d s.
\end{equation*}
The idea is to show that the second term essentially contributes to the exponential growth of the second moment, provided that the first term does not decay too fast with time. When the initial condition $u_0$ is bounded below, we immediately have the desired exponential growth since the first term is always bounded below.  But when $u_0$ have only positive support, the first term decays but only polynomially fast; this is Proposition \ref{estimatePDEpart}. And as time gets large, the "exponential growth" induced by the second term makes the second moment of the solution to start growing exponentially fast.  In some sense, this is an interplay between the "dissipative" effect of the fractional Laplacian and the noise term which is pumping energy to the system.  Proposition \ref{whitenoiselowbd} captures this interplay for the white noise driven equation and Proposition \ref{colourednoiselowerprop} for the coloured noise driven equation.
Finally, we mention that the situation is entirely different if one deals with equations on bounded domains with say the Dirichlet boundary conditions. In this case, there is no analogue of Proposition \ref{estimatePDEpart} and whether there is exponential growth of the second moment is highly dependent on the value of $\lambda$. A forthcoming paper \cite{FoonNual} will address this question.

We end this introduction with a plan of the article.  In the section 2, we give some estimates which will be needed for the proofs of the main results. In section 3, we prove the main result concerning the white noise drive equation. In section 4, we provide proofs of the main results for the coloured noise driven equation. Finally in section 5, we indicate some possible extensions.  The letter $c$ with or without superscripts or subscripts will denote constants whose exact values are not important and might vary from place to place.
\section{Preliminaries}

We begin this section with some information about the heat kernel of stable processes. Let $p_t(x, \, y)$ be the transition density for the $\alpha$-stable process on $\R^d$. 
We have the following bounds.

\begin{equation}
\label{estimatep}
c_1\left(t^{-d/\alpha}\wedge \frac{t}{|x-y|^{d+\alpha}}\right)\leq p_t(x,\,y)\leq c_2 \left(t^{-d/\alpha}\wedge \frac{t}{|x-y|^{d+\alpha}}\right),
\end{equation}
where $c_1$ and $c_2$ are positive constants.  Recall that 
\begin{equation*}
(\sG u)_t(x):=\int_{\R^d}p_t(x,\,y)u_0(y)\,\d y,
\end{equation*}
we then have the following proposition. This result may not be original, but, to our best knowledge, no such an estimate can be found in the existing literature. 
\begin{proposition} \label{estimatePDEpart}
There exists a $T>0$ and a constant $c_1$ such that for all $t>T$ and all $x\in B(0,\,t^{1/\alpha})$, we have
\begin{equation*}
(\sG u)_t(x)\geq \frac{c_1}{t^{d/\alpha}}.
\end{equation*}
\end{proposition}
\begin{proof}
We begin by writing 
\begin{equation*}
\begin{aligned}
(\sG u)_t(x) &\geq \int_{B(x,\,t^{1/\alpha})}p_t(x,\,y)u_0(y)\,\d y\\
&\geq \frac{c_2}{t^{d/\alpha}}\int_{B(x,\,t^{1/\alpha}/2)\cap B(0,\,t^{1/\alpha})}u_0(y)\,\d y.
\end{aligned}
\end{equation*}
Since $x\in B(0,\,t^{1/\alpha})$, we can use our assumption on $u_0$ to find $T>0$ large enough so that for $t>T$, we have
\begin{equation*}
\int_{B(x,\,t^{1/\alpha}/2)\cap B(0,\,t^{1/\alpha})\cap A }u_0(y)\,\d y\geq c_3.
\end{equation*}
This proves the result.
\end{proof}
We have the following estimate which will be useful for establishing temporal continuity property of the solution.  Recall the Fourier transform of the heat kernel that

\begin{equation*}
\hat{p}_t(\xi):=\E e^{i\xi \cdot X_t}=e^{-t|\xi|^\alpha}.
\end{equation*}

\begin{proposition}\label{increment}
Let $q\in(0,\,\frac{\alpha-\beta}{2\alpha})$ and $h\in(0,\,1)$, we then have 
\begin{equation*}
\int_0^{t}\int_{\R^d}|\hat{p}_{t-s+h}(\xi)-\hat{p}_{t-s}(\xi)|^2\frac{1}{|\xi|^{d-\beta}}\,\d \xi \,\d s\leq c_1h^{2q},
\end{equation*}
for some constant $c_1$.
\end{proposition}

\begin{proof}
From the Fourier transform of the heat kernel,  
\begin{equation*}
|\hat{p}_{t+h-s}(\xi)-\hat{p}_{t-s}(\xi)|^2=e^{-2(t-s)|\xi|^\alpha}[e^{-h|\xi|^\alpha}-1]^2.
\end{equation*}
We have 
\begin{equation*}
\begin{aligned}
\int_0^{t}\int_{\R^d}&|\hat{p}_{t+h-s}(\xi)-\hat{p}_{t-s}(\xi)|^2\frac{1}{|\xi|^{d-\beta}}\d \xi \d s\\
&=\int_0^t\int_{\R^d}\frac{e^{-2(t-s)|\xi|^\alpha}[e^{-h|\xi|^\alpha}-1]^2}{|\xi|^{d-\beta}}\,\d \xi \,\d s.
\end{aligned}
\end{equation*}
We use the following observation $|e^{-h|\xi|^\alpha}-1|\leq h^q|\xi|^{\alpha q}$ to bound the above term by 
\begin{equation*}
\begin{aligned}
h^{2q}\int_0^t\int_{\R^d}\frac{e^{-2(t-s)|\xi|^\alpha}|\xi|^{2\alpha q}}{|\xi|^{d-\beta}}\,\d \xi \,\d s.
\end{aligned}
\end{equation*}
We now separate the integral into two parts
\begin{equation*}
\begin{aligned}
\int_0^t\int_{\R^d}&\frac{e^{-2(t-s)|\xi|^\alpha}|\xi|^{2\alpha q}}{|\xi|^{d-\beta}}\,\d \xi \,\d s\\
&=\int_0^t\int_{|\xi|<1}\frac{e^{-2(t-s)|\xi|^\alpha}|\xi|^{2\alpha q}}{|\xi|^{d-\beta}}\,\d \xi \,\d s\\
&+\int_0^t\int_{|\xi|\geq 1}\frac{e^{-2(t-s)|\xi|^\alpha}|\xi|^{2\alpha q}}{|\xi|^{d-\beta}}\,\d \xi \,\d s.
\end{aligned}
\end{equation*}
The first integral appearing on the right hand side of the above display is clearly bounded. We need a bit more work for the second integral.
\begin{equation*}
\begin{aligned}
\int_0^t\int_{|\xi|\geq 1}&\frac{e^{-2(t-s)|\xi|^\alpha}|\xi|^{2\alpha q}}{|\xi|^{d-\beta}}\,\d \xi \,\d s\\
&\leq \int_{|\xi|\geq 1}\frac{1}{|\xi|^{d+\alpha-\beta-2\alpha q}}\d \xi.\\  
\end{aligned}
\end{equation*}
Since we are assuming that $q\leq \frac{\alpha-\beta}{2\alpha}$, the above integral is finite.  We now combine all the above estimates to obtain the result.
\end{proof}

In what follows, we will need the Gamma function which will be denoted $\Gamma(\cdot)$.  As usual $\Z_+$  and $\N$ denote the set of all positive integers and the set of all nonnegative integers respectively.  The next lemma is proved in \cite{Khoshnevisan:2013aa}, but we give a slightly different proof here.

\begin{lemma} \label{sumlowerbd}
Let $0 < \rho \leq 1$, then there exists a positive constant $c_1$ such that for all $b \geq (e/\rho)^{\rho}$,
\begin{equation*}
\sum_{j=0}^{\infty} \left( \frac{b}{j^{\rho}} \right)^j  \geq  \exp \left(c_1 b^{1/\rho}\right).
\end{equation*}  
\end{lemma}

\begin{proof}
We begin by writing
\begin{equation}
\label{sum2parts}
\begin{aligned}
\sum_{j=0}^{\infty} \left( \frac{b}{j^{\rho}} \right)^j &=  1 + \sum_{j\in \Z_+, ~ j\rho<1} \left( \frac{b}{j^{\rho}} \right)^j + \sum_{j\in \Z_+, ~ j\rho \geq1}\left( \frac{b}{j^{\rho}} \right)^j \\
&\geq 1 + \sum_{j\in \Z_+, ~ j\rho \geq1}\left( \frac{b}{j^{\rho}} \right)^j .
\end{aligned}
\end{equation}
We now use the well known fact that for $j \rho \geq 1$, $(j \rho/e )^{j \rho} \leq \Gamma(j \rho + 1)$ to bound the last term of (\ref{sum2parts}) as follows,
\begin{equation}
\label{gammaexpression}
\sum_{j\in \Z_+, ~j\rho \geq1}\left( \frac{b}{j^{\rho}} \right)^j \geq \sum_{j\in \Z_+, ~j\rho \geq1} \frac{\left(b^{1/\rho} \left( \frac{\rho}{e}\right)\right)^{j \rho}}{\Gamma(j \rho + 1)}.
\end{equation}
Recalling that  $0 < \rho \leq 1$ and $j \in \Z_+$, for each positive integer $k \geq 2 $, we can always find a distinct product  $j \rho$ such that $\Gamma (j \rho + 1) \leq \Gamma(\overline{j \rho} + 1 ) =  (\overline{j \rho})! = k!$ and $j \rho \geq 1$. Here $\overline{j\rho}$ denotes the smallest integer greater than $j\rho$.  We will substitute this into the right hand side of (\ref{gammaexpression}).
Since $b \geq (e/\rho)^{\rho}$, we have $b^{1/\rho}(\rho / e) \geq 1$.  Denote $\underline{j\rho}$ to be the greatest integer less than $j\rho$, we thus have 
 \begin{equation*}
\begin{aligned}
\sum_{j \in \Z_+, ~ j\rho \geq1} &\frac{\left(b^{1/\rho} \left( \frac{\rho}{e}\right)\right)^{j \rho}}{\Gamma(j \rho + 1)} 
\geq \sum_{j \in \Z_+, ~j\rho \geq1} \frac{\left(b^{1/\rho} \left( \frac{\rho}{e}\right)\right)^{\underline{j \rho}}}{(\overline{j \rho})!} \\
&\geq  \sum_{k=2}^{\infty} \frac{\left(b^{1/\rho} \left( \frac{\rho}{e}\right)\right)^{k-1}}{k!} =      \sum_{k=1}^{\infty} \frac{\left(b^{1/\rho} \left( \frac{\rho}{e}\right)\right)^{k}}{(k+1)!} \\
&\geq   \sum_{k=1}^{\infty} \frac{2^{-k}\left(b^{1/\rho} \left( \frac{\rho}{e}\right)\right)^{k}}{k!} 
= \exp  \left( b^{1/\rho} \left( \frac{\rho}{2 e}\right)\right) - 1.
\end{aligned}
\end{equation*}
Substituting this into (\ref{sum2parts}) completes the proof.
\end{proof}
The next result essentially reverses the inequality proved in the above lemma.  The proof will use some of the notations introduced in the proof of the previous lemma.
\begin{lemma} \label{MLestimate}
Let $\rho \in (0,\, 1]$, then there exist constants $c_1$ and $c_2$ such that 
\begin{equation*}
\sum_{j=0}^{\infty} \frac{  b ^{j}}{ \Gamma(j \rho + 1)} \leq c_1 \exp(c_2b^{1/ \rho})\quad\text{for\,all}\quad b>0.
\end{equation*}
\end{lemma}

\begin{proof}
We start by writing 
\begin{equation} \label{sumupperbd}
\sum_{j=0}^{\infty} \frac{ \left( b^{1/ \rho} \right)^{j \rho}}{ \Gamma(j \rho + 1)} = \sum_{j \in \N, ~ j \rho < 1} \frac{ \left( b^{1/ \rho} \right)^{j \rho}}{ \Gamma(j \rho + 1)} + \sum_{j \in \N, ~ j \rho \geq 1} \frac{ \left( b^{1/ \rho} \right)^{j \rho}}{ \Gamma(j \rho + 1)}. 
\end{equation}
We consider the case $b \in (0,1)$ first. 
\begin{equation*}
\begin{aligned}
\sum_{j \in \N, ~ j \rho \geq 1} \frac{ \left( b^{1/ \rho} \right)^{j \rho}}{ \Gamma(j \rho + 1)} &\leq \sum_{j \in \N, ~ j \rho \geq 1} \frac{ \left( b^{1/ \rho} \right)^{\underline{j \rho}}}{ (\underline{j \rho}) !} \\
&\leq c_1 \sum_{k=1}^{\infty} \frac{ \left( b^{1/ \rho} \right)^{k}}{ k !}\\
&\leq c_2 \left( \exp (b^{1 / \rho}) - 1 \right),
\end{aligned}
\end{equation*}
and
\begin{equation*}
\sum_{j \in \N, ~ j \rho < 1} \frac{ \left( b^{1/ \rho} \right)^{j \rho}}{ \Gamma(j \rho + 1)} \leq c_3,
\end{equation*}
where we have $\Gamma(j \rho + 1)$ is bounded below by a constant for all $j \rho <1$. Substituting them back into (\ref{sumupperbd}) yields
\begin{equation*}
\sum_{j=0}^{\infty} \frac{ \left( b^{1/ \rho} \right)^{j \rho}}{ \Gamma(j \rho + 1)} \leq c_4 \left( \exp (b^{1 / \rho}) +1 \right) \leq c_5  \exp (b^{1 / \rho}). 
\end{equation*}
\medskip
\par \noindent
We now turn our attention to the case $b \geq 1$ for which we have
\begin{equation*}
\begin{aligned}
\sum_{j \in \N, ~ j \rho \geq 1} \frac{ \left( b^{1/ \rho} \right)^{j \rho}}{ \Gamma(j \rho + 1)} &\leq \sum_{j \in \N, ~ j \rho \geq 1} \frac{ \left( b^{1/ \rho} \right)^{\overline{j \rho}}}{ (\underline{j \rho}) !} \\
&=  \sum_{j \in \N, ~ j \rho \geq 1} \frac{ \left( b^{1/ \rho} \right)^{\underline{j \rho} + 1}}{ (\underline{j \rho}) !} 
\end{aligned}
\end{equation*}
By relabelling the indices, we see that the above sum is bounded by 
\begin{equation*}
\begin{aligned}
c_6 \sum_{k=1}^{\infty}  \frac{ \left( b^{1/ \rho} \right)^{k+1}}{ k !} 
&\leq  c_7\sum_{k=2}^{\infty}  \frac{ \left( 2 b^{1/ \rho} \right)^{k}}{ k !} \\
&\leq c_8\left( \exp(2 b^{1/ \rho}) - 2 b^{1/ \rho} - 1 \right).
\end{aligned}
\end{equation*}
We also have
\begin{equation*}
\sum_{j \in \N, ~ j \rho < 1} \frac{ \left( b^{1/ \rho} \right)^{j \rho}}{ \Gamma(j \rho + 1)} \leq 2 c_9 b^{1/ \rho}.
\end{equation*}
Substituting them back into (\ref{sumupperbd}) yields
\begin{equation*}
\sum_{j=0}^{\infty} \frac{ \left( b^{1/ \rho} \right)^{j \rho}}{ \Gamma(j \rho + 1)} \leq c_{10}  \exp(c_{11} b^{1/ \rho}),
\end{equation*}
which is exactly what we wanted to prove.
\end{proof}

We now present some results concerning renewal inequalities. The proof is very similar to those proved in \cite{FoondunTianLiu}. The difference is that here we want bounds on the functions involved rather than finding their asymptotic properties as was the case in \cite{FoondunTianLiu}.  We will only sketch the proof. More results about renewal inequalities can be found in references such as \cite{Henry}.
\begin{proposition} \label{renewalineq}
Let $\rho > 0$ and suppose that $f(t)$ is a locally integrable function satisfying
\begin{equation*}
f(t) \leq c_1 + \kappa \int_0^t (t - s)^{\rho - 1} f(s) ds ~~~\text{for all}~~~ t>0 ,
\end{equation*}
where $c_1$ is some positive number. Then, we have 
\begin{equation*}
f(t) \leq c_2 \exp (c_3 (\Gamma (\rho))^{1 / \rho} \kappa^{1 / \rho} t) \quad\text{for all}\quad t>0,
\end{equation*}
for some positive constants $c_2$ and $c_3$.
\end{proposition}

\begin{proof}
The iterative procedure of Proposition 2.5 in \cite{FoondunTianLiu} yields
\begin{equation*}
f(t) \leq c_1 \sum_{k=0}^{\infty} \frac{ ( \Gamma (\rho) \kappa t^{\rho} ) ^k}{\Gamma(k \rho + 1)}.
\end{equation*}
Applying Lemma \ref{MLestimate} with $b = \Gamma (\rho) \kappa t^{\rho} $ proves the result.
\end{proof}

We now present the converse of the above result.

\begin{proposition}\label{renewalineq2}
Let $\rho > 0$ and suppose that $f(t)$ is a nonnegative, locally integrable function satisfying
\begin{equation*}
f(t) \geq c_1 + \kappa \int_0^t (t - s)^{\rho - 1} f(s) ds ~~~\text{for all}~~~ t>0 ,
\end{equation*}
where $c_1$ is some positive number. Then, we have 
\begin{equation*}
f(t) \geq c_2 \exp (c_3 (\Gamma (\rho))^{1 / \rho} \kappa^{1 / \rho} t) \quad\text{for all}\quad t>\frac{e}{\rho}(\Gamma(\rho)\kappa)^{-1/\rho},
\end{equation*}
for some positive constants $c_2$ and $c_3$.
\end{proposition}
\begin{proof}
As in the proof of Proposition 2.6 of \cite{FoondunTianLiu}, we have 
\begin{equation*}
f(t) \geq c_1 \sum_{k=0}^{\infty} \frac{ ( \Gamma (\rho) \kappa t^{\rho} ) ^k}{\Gamma(k \rho + 1)}.
\end{equation*}
From the arguments used in the proof of Lemma \ref{sumlowerbd}, we have the desired result once we choose $b = \Gamma (\rho) \kappa t^{\rho} $.
\end{proof}

\section{Proof of Theorem \ref{theoremwhite}.}
We begin with a calculus lemma, which won't be needed for the proof of Theorem \ref{theoremwhite}, but required for later on.
\begin{lemma}\label{integrals}
For any $t>0$ and $\alpha>1$, we have 
\begin{equation*}
\begin{aligned}
\int_0^t\int_0^{s_1}\cdots \int_0^{s_{k-1}}&\frac{1}{[(t-s_1)(s_1-s_2)\cdots(s_{k-1}-s_k)]^{1/\alpha}}\d s_{k}\d s_{k-1}\cdots \d s_{1}\\
&\geq c_1^k\left(\frac{t}{k}\right)^{k(\alpha-1)/\alpha},
\end{aligned}
\end{equation*}
where $c_1$ is a constant independent of $k$.
\end{lemma}

\begin{proof}

We begin by reducing the region of integration as follows,
\begin{equation*}
\begin{aligned}
\int_0^t&\int_0^{s_1}\cdots \int_0^{s_{k-1}}\frac{1}{[(t-s_1)(s_1-s_2)\cdots(s_{k-1}-s_k)]^{1/\alpha}}\d s_{k}\d s_{k-1}\cdots \d s_{1}\\
&\geq \int_{t-t/k}^{t}\int_{s_1-t/k}^{s_1}\cdots \int_{s_{k-1}-t/k}^{s_{k-1}}\frac{1}{[(t-s_1)(s_1-s_2)\cdots(s_{k-1}-s_k)]^{1/\alpha}}\d s_{k}\d s_{k-1}\cdots \d s_{1}.
\end{aligned}
\end{equation*}
We set $s_0:=t$ and make the substitution $\tilde{s}_i=s_{i-1}-s_i$ for $i=1,\cdots, k$. This together with some calculus completes the proof.
\end{proof}

We now present a key result. As mentioned in the introduction, this quantifies the relationship between the ``dissipative" effect of the fractional Laplacian and the ``growth" induced by the noise term. It will be clear that the proof relies heavily on the heat kernel estimates for the fractional Laplacian.  

Fix $t>0$ and set
\begin{equation*}
g_t:=\inf_{x\in B(0,\,t^{1/\alpha})}(\sG u)_t(x).
\end{equation*}
We then have the following.

\begin{proposition}
\label{whitenoiselowbd}
Let $t>0$, then for all $x\in B(0, t^{1/\alpha})$ the following holds
\begin{equation*}
\E|u_t(x)|^2\geq g_t^2\sum_{k=0}^\infty \left(\frac{l_\sigma^2\lambda^2\alpha}{\alpha-1} \right)^k\left(\frac{t}{k}\right)^{k(\alpha-1)/\alpha}.
\end{equation*}
\end{proposition}
\begin{proof}
We start off with the mild formulation of the solution, take the second moment and then use the lower bound on $\sigma$ to write
\begin{equation*}
\E|u_t(x)|^2\geq |(\sG u)_t(x)|^2+ l_\sigma^2\lambda^2\int_0^t\int_{\R}p_{t-s}^2(x,\,y)\E|u_s(y)|^2\d s\d y,
\end{equation*}
which immediately gives 
\begin{equation*}
\begin{aligned}
\E|u_t(x)|^2&\geq |(\sG u)_t(x)|^2+ l_\sigma^2\lambda^2\int_0^t\int_{\R}p_{t-s}^2(x,\,y)\E|u_s(y)|^2\d s\,\d y\\
&\geq |(\sG u)_t(x)|^2+ l_\sigma^2\lambda^2\int_0^t\int_{\R}p_{t-s_1}^2(x,\,z_1)|(\sG u)_{s_1}(z_1)|^2\d s_1\,\d z_1\\
&+l_\sigma^4\lambda^4\int_0^t\int_{\R}\int_0^{s_1}\int_{\R}p_{t-s_1}^2(x,\,z_1)p_{s_1-s_2}^2(z_1,\,z_2)\E|u_{s_2}(z_2)|^2\d s_2\,\d z_2\,\d s_1\,\d z_1.
\end{aligned}
\end{equation*}
And if we use the following 
\begin{equation*}
\E|u_{s_i}(z_i)|^2\geq |(\sG u)_{s_i}(z_i)|^2+ l_\sigma^2\int_0^{s_i}\int_{\R}p_{s_i-s_{i+1}}^2(z_i,\,z_{i+1})\E|u_{s_{i+1}}(z_{i+1})|^2\d s_{i+1}\d z_{i+1},
\end{equation*}
we end up with 
\begin{equation*}
\begin{aligned}
\E|u_t(&x)|^2\\
&\geq |(\sG u)_t(x)|^2\\
&+\sum_{k=1}^\infty (\lambda l_\sigma)^{2k}\int_0^t\int_{\R}\int_0^{s_1}\int_{\R}\dots\int_0^{s_{k-1}}\int_{\R} |(\sG u)_{s_k}(z_k)|^2\\
&\prod_{i=1}^{k}p^2_{s_{i-1}-s_i}(z_{i-1},z_i)\,\d z_{k+1-i}\,\d s_{k+1-i},
\end{aligned}
\end{equation*}
where we have used the convention that $s_0:=t$ and $z_0:=x$. We now restrict all $z_k$ such that $z_k \in B(0,\,t^{1/\alpha})$ for all $k \geq 0$ to obtain

\begin{equation*}
\begin{aligned}
\E|u_t(x)|^2 &\geq g_t^2\\
&+g_t^2\sum_{k=1}^\infty (\lambda l_\sigma)^{2k}\int_0^t\int_\R \int_0^{s_1}\int_\R \dots\int_0^{s_{k-1}}\int_{B(0,\,t^{1/\alpha})} \\
&\prod_{i=1}^kp^2_{s_{i-1}-s_i}(z_{i-1},z_i)\,\d z_{k+1-i}\, \d s_{k+1-i}.
\end{aligned}
\end{equation*}
We now shrink the temporal region of integration and make a proper change of variable, as in the proof of Lemma \ref{integrals}, we end up with
\begin{equation*}
\begin{aligned}
\E|u_t(x)|^2 &\geq g_t^2\\
&+g_t^2\sum_{k=1}^\infty (\lambda l_\sigma)^{2k}\int_0^{t/k}\int_\R \int_0^{t/k}\int_\R \dots\int_0^{t/k}\int_{B(0,\,t^{1/\alpha})} \\
&\prod_{i=1}^kp^2_{s_i}(z_{i-1},z_i)\,\d z_{k+1-i}\, \d s_{k+1-i}.
\end{aligned}
\end{equation*}
We will further restrict the spatial domain of integration by appropriately choosing the points $\{z_i\}_{i=1}^k$ such that
\begin{equation*}
z_i\in B(z_{i-1}, s_i^{1/\alpha})\cap B(0,\,t^{1/\alpha}).
\end{equation*}
Now since $|z_i-z_{i-1}|\leq s_i^{1/\alpha}$, we have $p_{s_i}(z_{i-1},z_i)\geq c_1s_i^{-1/\alpha}$.  For notational convenience, we set $\sA_i:=|B(z_{i-1}, s_i^{1/\alpha})\cap B(0,\,t^{1/\alpha})|$. We clearly have $|\sA_i|\geq c_2 s_i^{1/\alpha}$.
We now use the heat kernel estimates and the above to write
\begin{equation*}
\begin{aligned}
\int_{\R\times\R\times\dots\times B(0,\,t^{1/\alpha})}&\prod_{i=1}^k p^2_{s_i}(z_{i-1},z_i)\,\d z_i\\
&\geq \int_{\sA_1\times\sA_2\times\dots\times \sA_k}\prod_{i=1}^k p^2_{s_i}(z_{i-1},z_i)\,\d z_i\\
&\geq c_3^k\prod_{i=1}^k\frac{1}{s_i^{1/\alpha}}.
\end{aligned}
\end{equation*}
We therefore have 
\begin{equation*}
\begin{aligned}
\int_0^{t/k}\int_0^{t/k}\dots\int_0^{t/k}&c_3^k\prod_{i=1}^k\frac{1}{s_i^{1/\alpha}}\d s_k\,\d s_{k}\dots\,\d s_1\\
&\geq c_3^k \left(\frac{t}{k}\right)^{k(\alpha-1)/\alpha}.
\end{aligned}
\end{equation*}
Combining the above estimates, we obtain
\begin{equation*}
\begin{aligned}
\E|u_t(x)|^2\geq g_t^2+g_t^2\sum_{k=1}^\infty(c_3l_\sigma^2\lambda^2)^k\left(\frac{t}{k} \right)^{k(\alpha-1)/\alpha}.
\end{aligned}
\end{equation*}
We have thus proved the result.
\end{proof}

{\it Proof of Theorem \ref{theoremwhite}.}
Using Lemma \ref{sumlowerbd}, we have the first statement of the theorem. For the second part of theorem, we fix $x\in \R$. Clearly we have $x\in B(0,\,2|x|)$ and by the first part of the theorem, we have for $t^{1/\alpha}\geq 2|x|\vee T$, 
\begin{equation*}
\E|u_t(x)|^2\geq c \exp \left( c' \lambda^{2 \alpha/(\alpha - 1)}t \right).
\end{equation*}
By taking the appropriate limit, we obtain the second part of the theorem.
\qed

\begin{remark}
We now use Lemma \ref{integrals} to see how the proof of the above result simplifies when the initial function is assumed to be bounded below. We start with 
\begin{equation*}
\begin{aligned}
\E|u_t(&x)|^2\\
&\geq |(\sG u)_t(x)|^2\\
&+\sum_{k=1}^\infty (\lambda l_\sigma)^{2k}\int_0^t\int_{\R}\int_0^{s_1}\int_{\R}\dots\int_0^{s_{k-1}}\int_{\R} |(\sG u)_{s_k}(z_k)|^2\\
&\prod_{i=1}^kp^2_{s_{i-1}-s_i}(z_{i-1},z_i)\,\d z_{k+1-i} \, \d s_{k+1-i}.
\end{aligned}
\end{equation*}
Since the initial function is bounded below we will have $(\sG u)_{s_k}(z_k)\geq c_1$ for some constant $c_1$. We now look that following iterated integral 
\begin{equation*}
\int_0^t\cdots \int_0^{s_{k-1}} \int_\R\cdots\int_\R \prod_{i=1}^kp^2_{s_{i-1}-s_i}(z_{i-1},z_i)\,\d z_{k+1-i} \, \d s_{k+1-i}.
\end{equation*}
We now use the semigroup and Lemma \ref{integrals} property to reduce the above quantity to reduce
\begin{equation*}
\begin{aligned}
\int_0^t&\cdots\int_0^{s_{k-1}}\prod_{i=1}^kp_{s_{i-1}-s_i}(0,0)\d s_i\\
&\geq\int_0^t\cdots\int_0^{s_{k-1}}\prod_{i=1}^k\frac{1}{(s_{i-1}-s_i)^{1/\alpha}}\d s_i\\
&\geq c_1\left(\frac{t}{k}\right)^{k(\alpha-1)/\alpha}.
\end{aligned}
\end{equation*}
Combining all the above estimates together we obtain

\begin{equation*}
\begin{aligned}
\E|u_t(x)|^2\geq c_3\sum_{k=0}^\infty(c_3l_\sigma^2\lambda^2)^k\left(\frac{t}{k} \right)^{k(\alpha-1)/\alpha}.
\end{aligned}
\end{equation*}
We have included this to illustrate the fact that when the initial condition is bounded below, one can use the semigroup properties of the heat kernel and obtain a similar result. This also highlights the technical issues we run into when the initial condition is not bounded below.  

The above gives exponential bounds for the the second moment for $t>T$. What about for $t\in(0,T]$?  When the initial condition is a function which is bounded below, we quite easily get the required bound. As before, we have

\begin{equation*}
\begin{aligned}
\E|u_t(x)|^2&=(\sG u)_t(x)^2+\lambda^2\int_0^t\int_{\R}p^2_{t-s}(x-y)\E|\sigma(u_s(y))|^2\,\d y\,\d s\\
&=I_1+I_2.
\end{aligned}
\end{equation*}
Clearly $I_1\geq c_1$. For $I_2$, we have 
\begin{equation*}
\begin{aligned}
I_2&\geq \lambda^2l_\sigma^2\int_0^t\inf_{y\in \R}\E|u_s(y)|^2\int_{\R}p^2_{t-s}(x,\,y)\,\d y\,\d s\\
&\geq c_2\lambda^2l_\sigma^2\int_0^t\frac{\inf_{y\in \R}\E|u_s(y)|^2}{(t-s)^{1/\alpha}}\,\d s.
\end{aligned}
\end{equation*}
Putting these estimates together we have 
\begin{equation}\label{all times}
\inf_{x\in \R}\E|u_t(x)|^2\geq c_1+c_2\lambda^2l_\sigma^2\int_0^t\frac{\inf_{y\in \R}\E|u_s(y)|^2}{(t-s)^{1/\alpha}}\,\d s.
\end{equation}
If $t<T$, for some large $T$, the above inequality reduces to 
\begin{equation*}
\inf_{x\in \R}\E|u_t(x)|^2\geq c_1+\frac{c_2\lambda^2l_\sigma^2}{T^{1/\alpha}}\int_0^t \inf_{y\in \R}\E|u_s(y)|^2\d s,
\end{equation*}
which gives the required exponential bound. For $t>T$, \eqref{all times} together with Proposition \ref{renewalineq2} gives the required bound with the correct rate with respect of $\lambda$.  We have thus given two different ways of proving exponential bounds when the initial condition is bounded below.  These work mainly because we have explicit heat kernel estimates for the fractional Laplacian.  This was not the case in \cite{FK}.  We will make a similar remark concerning the coloured noise equation later.

\end{remark}

\section{Proofs of Theorems \ref{colourednoisetheorem}-\ref{continuity}.}
We start this section with the following estimate. Recall that $f$ denotes the correlation function of the coloured noise.
\begin{lemma} \label{IntegrlBound}
For any $t > 0$ and all $x,\, y \in \R^d$, there exists some positive constant $c_1$ such that
\begin{equation*}
\int_{\R^d \times \R^d} p_t(x, \, \omega) p_t(y, \, z) f(\omega, \, z) \d \omega \d z \leq \frac{c_1}{t^{\beta / \alpha}}.
\end{equation*}
\end{lemma}
\begin{proof}
Since
\begin{equation*}
\int_{\R^d} \int_{\R^d}  p_t(x, \, \omega) p_t(y, \, z) f(\omega, \, z) \d \omega \d z \leq \int_{\R^d} p_{2t} (\omega, \, x-y) f(\omega, \, 0) \d \omega,
\end{equation*}
the scaling property of the heat kernel and a proper change of the variable prove the result.
\end{proof}

Set
\begin{equation}\label{sup}
F(t):= \sup_{x\in \R^d}\E|u_t(x)|^2,
\end{equation}
where $u_t$ denotes the unique solution to \eqref{main-eq}.  We then have the following.

\begin{proposition} \label{colourednoiseupperprop}
There exist constants $c_1$ and $c_2$ such that for all $t>0$, we have 
\begin{equation*}
F(t) \leq c_1+ c_2(\lambda L_\sigma)^2 \int_0^t \frac{F(s)}{(t -s)^{\beta / \alpha}} \d s.
\end{equation*}
\end{proposition}
\begin{proof}

We start with the mild formulation given by (\ref{mild-main}), then take the second moment to obtain the following
\begin{equation*}
\begin{aligned}
\E|&u_t(x)|^2 =|(\sG u)_t(x)|^2 \\
&~+ \lambda^2 \int_0^t\int_{\R^d \times \R^d} p_{t-s}(x,\,y) p_{t-s} (x, \, z) f(y, \, z) \E[\sigma(u_s(y))\sigma(u_s(z))]\d y \, \d z \, \d s\\
&=I_1+I_2.
\end{aligned}
\end{equation*}  
We begin by looking at the first term $I_1$. Since $u_0(x)$ is bounded, we have $I_1\leq c_3$.  We now use the assumption on $\sigma$ together with H\"older's inequality to see that 

\begin{equation*}
\begin{aligned}
\E [\sigma(u_s(y))\sigma(u_s(z))]&\leq L_\sigma^2\E [|u_s(y) u_s(z)|]\\
&\leq [\E |u_s(y) |^2]^{1/2} [\E |u_s(z) |^2]^{1/2}\\
&\leq \sup_{x \in \R^d} \E |u_s(x) |^2
\end{aligned}
\end{equation*}
Therefore using the notation (\ref{sup}) as well as Lemma \ref{IntegrlBound}, the second term $I_2$ can be bounded as follows.
\begin{equation*}
\begin{aligned}
I_2\leq c_5(\lambda L_\sigma)^2\int_0^t\frac{F(s)}{(t-s)^{\beta/\alpha}}\,\d s.
\end{aligned}
\end{equation*}
Combining the above estimates, we obtain the required result.
\end{proof}
 Recall that
\begin{equation*}
g_t:=\inf_{x\in B(0,\,t^{1/\alpha})}(\sG u)_t(x).
\end{equation*}
We then have the following.
\begin{proposition} \label{colourednoiselowerprop}
Given any $x \in \R^d$, for all $t$ satisfying $x \in B(0,\,t^{1/\alpha})$, there exists some constant $c_1$ such that
\begin{equation*}
\E|u_t(x)|^2\geq g_t^2\sum_{k=0}^\infty (c_1\lambda l_\sigma)^{2k}\left(\frac{t}{k} \right)^{k(\alpha-\beta)/\alpha}
\end{equation*}
holds.
\end{proposition}
\begin{proof}
We start off with the mild formulation of the solution, take the second moment and use the lower bound on $\sigma$ to end up with
\begin{equation*}
\begin{aligned}
\E|u_t(&x)|^2\\
&\geq |(\sG u)_t(x)|^2+\lambda^2l_\sigma^2\int_0^t\int_{\R^d}\int_{\R^d}\\
& p_{t-s_1}(x,\,z_1)p_{t-s_1}(x,\,z_1')f(z_1,\,z_1')\E|u_{s_1}(z_1)u_{s_1}(z_1')|\,\d z_1\,\d z'_1 \, \d s_1.
\end{aligned}
\end{equation*}
We now have 
\begin{equation*}
\begin{aligned}
\E|u_{s_1}&(z_1)u_{s_1}(z_1')|\\
&\geq \E u_{s_1}(z_1)u_{s_1}(z_1')\\
&\geq (\sG u)_{s_1}(z_1)(\sG u)_{s_1}(z_1')+\lambda^2l_\sigma^2\int_0^{s_1} \int_{\R^d}\int_{\R^d}\\
& p_{s_1-s_2}(z_1,\,z_2)p_{s_1-s_2}(z_1',\,z_2')f(z_2,\,z_2')\E|u_{s_2}(z_2)u_{s_2}(z_2')|\,\d z_2\,\d z_2'\,\d s_2.
\end{aligned}
\end{equation*}
Applying the above recursively, we end up with the following
\begin{equation*}
\begin{aligned}
\E|u_t(&x)|^2\\
&\geq |(\sG u)_t(x)|^2+\sum_{k=1}^\infty (\lambda l_\sigma)^{2k}\int_0^t\int_{\R^d\times \R^d}\int_0^{s_1}\int_{\R^d\times \R^d}\dots \int_0^{s_{k-1}}\int_{\R^d\times \R^d}\\
& \bigg((\sG u)_{s_k}(z_k)(\sG u)_{s_k}(z'_k)\\
&\times \prod_{i=1}^kp_{s_{i-1}-s_i}(z_{i-1},\,z_i)p_{t-s}(z'_{i-1},\,z_i')f(z_i,\,z_i')\bigg)\,\d z_{k+1-i}\,\d z_{k+1-i}'\,\d s_{k+1-i}.
\end{aligned}
\end{equation*}
We now reduce the spatial domain of integration to end up with 
\begin{equation*}
\begin{aligned}
\E|u_t(&x)|^2\\
&\geq g_t^2+g_t^2\sum_{k=1}^\infty (\lambda l_\sigma)^{2k}\int_0^t\int_{\R^d\times\R^d}\int_0^{s_1}\int_{\R^d\times \R^d}\dots \int_0^{s_{k-1}}\int_{B(0,\,t^{1/\alpha})\times B(0,\,t^{1/\alpha})}\\
&\prod_{i=1}^kp_{s_{i-1}-s_i}(z_{i-1},\,z_i)p_{s_{i-1}-s_i}(z'_{i-1},\,z_i')f(z_i,\,z_i')\,\d z_{k+1-i}\,\d z_{k+1-i}'\,\d s_{k+1-i}.
\end{aligned}
\end{equation*}
As in the proof of Proposition \ref{whitenoiselowbd}, we reduce the temporal domain of integration and make an appropriate change of variable to end up with
\begin{equation*}
\begin{aligned}
\E|u_t(&x)|^2\\
&\geq g_t^2+g_t^2\sum_{k=1}^\infty (\lambda l_\sigma)^{2k}\int_0^{t/k}\int_{\R^d\times\R^d}\dots \int_0^{t/k}\int_{B(0,\,t^{1/\alpha})\times B(0,\,t^{1/\alpha})}\\
&\prod_{i=1}^kp_{s_i}(z_{i-1},\,z_i)p_{s_i}(z'_{i-1},\,z_i')f(z_i,\,z_i')\,\d z_{k+1-i}\,\d z_{k+1-i}'\,\d s_{k+1-i}.
\end{aligned}
\end{equation*}
Recall that $x\in B(0,\,t^{1/\alpha})$ and consider

\begin{equation*}
z_i\in B(x,\,s_1^{1/\alpha}/2)\cap B(z_{i-1}, s_i^{1/\alpha}),
\end{equation*}
and
\begin{equation*}
z_i'\in B(x,\,s_1^{1/\alpha}/2)\cap B(z_{i-1}', s_i^{1/\alpha}).
\end{equation*}
These imply that $|z_i-z_i'|\leq s_1^{1/\alpha}$ which gives $f(z_i, z_i')\geq s_1^{-\beta/\alpha}$. We also have $|z_i-z_{i-1}|\leq s_i^{1/\alpha}$ and $|z'_i-z'_{i-1}|\leq s_i^{1/\alpha}$ which imply that $p(s_i, z_{i-1},\,z_i)\geq c_1s_i^{-d/\alpha}$ and $p(s_i, z_{i-1},\,z_i)\geq c_1s_i^{-d/\alpha}$.  In other words, we are looking at the points $\{s_i,\,z_i,\,z_i' \}_{i=0}^k$ such that the following holds
\begin{equation*}
\prod_{i=1}^kp(s_i, z_{i-1},\,z_i)p(s_i, z'_{i-1},\,z'_i)f(z_i, z_i')\geq c_1^{2k}\prod_{i=1}^k\frac{1}{s_i^{2d/\alpha}s_1^{\beta/\alpha}}.
\end{equation*}

Now we have that $|B(x,\,s_1^{1/\alpha}/2)\cap B(z_{i-1}, s_i^{1/\alpha})|\geq c_2s_i^{d/\alpha}$, for some constant $c_2$.  For notational convenience, we set $\sA_i:=\{z_i\in B(x,\,s_1^{1/\alpha}/2)\cap B(z_{i-1}, s_i^{1/\alpha})\}$ and $\sA_i':=\{z_i'\in B(x,\,s_1^{1/\alpha}/2)\cap B(z'_{i-1}, s_i^{1/\alpha})\}$.

\begin{equation*}
\begin{aligned}
\int_0^{t/k}&\int_{\R^d \times \R^d}\dots \int_0^{t/k}\int_{B(0,\,t^{1/\alpha})\times B(0,\,t^{1/\alpha})}\\
&\prod_{i=1}^kp_{s_i}(z_{i-1},\,z_i)p_{s_i}(z'_{i-1},\,z_i')f(z_i,\,z_i')\,\d z_{k+1-i}\,\d z_{k+1-i}'\,\d s_{k+1-i}\\
&\geq \int_0^{t/k}\int_{\sA_1}\int_{\sA'_1}\dots\int_0^{t/k}\int_{\sA_k}\int_{\sA'_k}\\
&\prod_{i=1}^kp_{s_i}(z_{i-1},\,z_i)p_{s_i}(z'_{i-1},\,z_i')f(z_i,\,z_i')\,\d z_{k+1-i}\,\d z_{k+1-i}'\,\d s_{k+1-i}\\
&\geq c_2^k\int_0^{t/k}\dots\int_0^{t/k} \frac{1}{s_i^{k\beta/\alpha}}\,\d s_k\,\d s_{k-1}\dots \d s_1\\
&=c_2^k\left(\frac{t}{k} \right)^{k(\alpha-\beta)/\alpha}.
\end{aligned}
\end{equation*}
We now combine the above estimates to obtain,
\begin{equation*}
\E|u_t(x)|^2\geq g_t^2+g_t^2\sum_{k=1}^\infty (c_3\lambda l_\sigma)^{2k}\left(\frac{t}{k}\right)^{k(\alpha-\beta)/\alpha},
\end{equation*}
which proves the result.
\end{proof}
\noindent
Now we are ready to prove Theorem \ref{colourednoisetheorem}
\medskip \par \noindent
{\it Proof of Theorem \ref{colourednoisetheorem}.}
We prove the upper bound first. But this is an immediate consequence of Propositions \ref{colourednoiseupperprop} and  \ref{renewalineq} with $\rho = (\alpha - \beta)/ \alpha$ and $\kappa = (\lambda L_{\sigma})^2$. We now turn our attention to the lower bound. We note that from Proposition \ref{estimatePDEpart}, we have that $g_t\geq c_1t^{-d/\alpha}$ for $t>T$, where $T$ is some positive constant. We can use this together with Proposition \ref{colourednoiselowerprop} to write
\begin{equation*}
\begin{aligned}
\E|u_t(x)|^2&\geq g_t^2\sum_{k=0}^\infty (c_1\lambda l_\sigma)^{2k}\left(\frac{t}{k} \right)^{k(\alpha-\beta)/\alpha} \\
&\geq t^{-2d/\alpha} \sum_{k=0}^\infty (c_1\lambda l_\sigma)^{2k}\left(\frac{t}{k} \right)^{k(\alpha-\beta)/\alpha}.
\end{aligned}
\end{equation*}
By taking $T$ large enough and using Lemma \ref{sumlowerbd}, we obtain
\begin{equation*}
\E|u_t(x)|^2\geq  c_2\exp \left( c'_2 \lambda^{2 \alpha/(\alpha - \beta)}t \right)\quad\text{for all}\quad t>T,
\end{equation*}
for some constants $c_2$ and $c_2'$. The final part of the theorem follows easily.
\qed \\

\begin{remark}\label{correlated}
For the purpose of this remark, we will assume that the initial condition is bounded below. We will again show that one can get the exponential growth by using heat kernel estimates and Proposition \ref{renewalineq2}. Since we now have coloured noise, we need to develop a slightly different strategy. We seek to find a renewal inequality for the function $G(t)$ defined below. We first make the observation that there exists a positive constant $c_1$ such that 
\begin{equation*}
\int_{\R^d\times \R^d}p_{t-s}(x, w_1)p_{t-s}(x, w_2)f(w_1,w_2)\,\d w_1\,\d w_2\geq \frac{c_1}{(t-s)^{\beta/\alpha}}.
\end{equation*} 
We now set
\begin{equation*}
G(t):=\inf_{x,\,y\in \R^d}\E|u_t(x)u_t(y)|,
\end{equation*}
and use the mild formulation of the solution to obtain 
\begin{equation*}
G(t)\geq c_2+c_3\lambda^2\int_0^t\frac{G(s)}{(t-s)^{\beta/\alpha}}\,\d s.
\end{equation*}
It should be clear that we have used, in an essential way, the fact that the initial condition is bounded below. Now an application of Proposition \ref{renewalineq2} gives us $G(t)\geq c_4\exp \left( c_5 \lambda^{2 \alpha/ (\alpha - \beta)} t \right)$. This essentially shows that second moment grows exponentially as time gets large, which is what we wanted to prove.  The point of this remark is to show that in our setting, one can significantly simplify the proof of the exponential growth in \cite{FK2}, where Fourier techniques were used instead of heat kernel estimates. We have of course used a very specific correlation function, but it seems that this argument will work for a much larger class of correlation functions as well. To be more specific, in \cite {HuNuaSo}, the authors considered noises satisfying, 
\begin{equation}\label{assump}
\int_{\R^d\times \R^d}p_{t}(x, w_1)p_{t}(x, w_2)f(w_1,w_2)\,\d w_1\,\d w_2\geq \frac{c_1}{t^{\gamma}},
\end{equation} 
where $\gamma$ is some positive number. The above will apply to this wide class of noises as well. In fact,  since \eqref{assump} involves the heat kernel as well as the correlation function, we could take it as a condition on both the operator and the noise term. This will simplify the arguments of \cite{FK2} for an even greater class of equations than the one considered here. 
\end{remark}

{\it Proof of Theorem \ref{colourednoiselambdacoro}.}
From the upper bound in Theorem \ref{colourednoisetheorem}, we have that for any $x \in \R^d$
\begin{equation*}
\E |u_t(x)|^2 \leq c \exp \left( c' \lambda^{2 \alpha/ (\alpha - \beta)} t \right)\quad\text{for\,\,all}\quad t>0,
\end{equation*}
from which we easily have
\begin{equation*}
\limsup_{\lambda\rightarrow \infty}\frac{\log \log \E|u_t(x)|^2}{\log \lambda}\leq \frac{2\alpha}{\alpha-\beta}.
\end{equation*}

We will seek a converse to the above inequality. Fix $x\in \R^d$, if $t$ is already large enough so that $x\in B(0,\,t^{1/\alpha})$, then by Proposition \ref{colourednoiselowerprop} we have 
\begin{equation*}
\E|u_t(x)|^2\geq g_t^2\sum_{k=0}^\infty (c_3l_\sigma \lambda)^{2k}\left(\frac{t}{k}\right)^{k(\alpha-\beta)/\alpha},
\end{equation*}
which together with Lemma \ref{sumlowerbd} gives, 
\begin{equation*}
\liminf_{\lambda\rightarrow \infty}\frac{\log \log E|u_t(x)|^2}{\log \lambda}\geq \frac{2\alpha}{\alpha-\beta}.
\end{equation*}
Now if $x\notin B(0,\,t^{1/\alpha})$, we can choose a constant $\kappa>0$ so that $x\in B(0,\,\kappa t^{1/\alpha})$, we can use the ideas of Proposition \ref{colourednoiselowerprop} to end up with 
\begin{equation*}
\E|u_t(x)|^2\geq g_{\kappa t}^2\sum_{k=0}^\infty (c_4l_\sigma \lambda)^{2k}\left(\frac{t}{k}\right)^{k(\alpha-\beta)/\alpha}
\end{equation*}
and the result follows easily from that.
\qed

{\it Proof of Theorem \ref{excitation-index}.}
We begin by estimating an upper bound on $e(t)$. We start with the mild solution and take the second moment to obtain 
\begin{equation*}
\begin{aligned}
\sE_t(\lambda)^2&\leq \int_{\R^d}\left|\int_{\R^d}p_t(x,\,y)u_0(y)\,\d y\right|^2 \d x\\
&+ (\lambda L_\sigma)^2 \int_{\R^d}\int_0^t\int_{\R^d\times \R^d}p_{t-s}(x,\,y_1)p_{t-s}(x,\,y_2)f(y_1,\,y_2)\E[|u_s(y_1)u_tsy_2)|]\,\d y_1\d y_2\d s \d x \\
&=I_1+I_2.
\end{aligned}
\end{equation*}
Clearly $I_1\leq c_1$. We need to find a lower bound on $I_2$.
\begin{equation*}
\begin{aligned}
I_2&\leq \int_0^t\int_{\R^d\times \R^d}p_{2(t-s)}(y_1,\,y_2)f(y_1,\,y_2)\E[|u_sy_1)u_s(y_2)|]\,\d y_1\d y_2\d s\\
&\leq \int_0^t\sup_{y\in \R^d}\E|u_s(y)|^2\int_{\R^d\times \R^d}p_{2(t-s)}(y_1,\,y_2)f(y_1,\,y_2)\d y_1\d y_2\d s\\
&\leq c_2\int_0^t\frac{1}{(t-s)^{\beta/\alpha}}e^{c_3\lambda^{2\alpha/(\alpha-\beta)}s}\,\d s\\
&\leq c_4 e^{c_5\lambda^{2\alpha/(\alpha-\beta)}t}.
\end{aligned}
\end{equation*}
We therefore have 
\begin{equation*}
\limsup_{\lambda\rightarrow \infty}\frac{\log \log \sE(t)}{\log \lambda}\leq \frac{2\alpha}{\alpha-\beta}.
\end{equation*}
We now seek a lower bound on $e(t)$. As in the proof of Proposition \ref{colourednoiselowerprop}, we have 
\begin{equation*}
\begin{aligned}
\E|u_t(&x)|^2\\
&\geq |(\sG u)_t(x)|^2+\sum_{k=1}^\infty (\lambda l_\sigma)^{2k}\int_0^t\int_{\R^d\times \R^d}\int_0^{s_1}\int_{\R^d\times \R^d}\dots \int_0^{s_{k-1}}\int_{\R^d\times \R^d}\\
&\bigg((\sG u)_{s_k}(z_k)(\sG u)_{s_k}(z'_k)\\
&\times \prod_{i=1}^kp_{s_{i-1}-s_i}(z_{i-1},\,z_i)p_{t-s}(z'_{i-1},\,z_i')f(z_i,\,z_i')\bigg) \,\d z_{k+1-i}\,\d z_{k+1-i}' \,\d s_{k+1-i}.
\end{aligned}
\end{equation*}
We now integrate both sides with respect to $x$ and use the techniques employed in the proof of Proposition \ref{colourednoiselowerprop} to obtain 
\begin{equation*}
\sE(t)^2\geq c_1+c_2\sum_{k=1}^\infty (c_3\lambda l_\sigma)^{2k}\left(\frac{t}{k}\right)^{k(\alpha-\beta)/\alpha}.
\end{equation*}
This together with Lemma \ref{sumlowerbd} yields 
\begin{equation*}
\liminf_{\lambda\rightarrow \infty}\frac{\log \log \sE(t)}{\log \lambda}\geq \frac{2\alpha}{\alpha-\beta}.
\end{equation*}
This proves the theorem.
\qed

{\it Proof of Theorem \ref{continuity}.}
As usual, the proof makes use of Kolmogorov's continuity theorem. We will therefore look at the increment $\E|u_{t+h}(x)-u_t(x)|^p$ for $h\in(0,\,1)$ and $p\geq 2$. We have 

\begin{equation*}
\begin{aligned}
u_{t+h}(x)&-u_t(x)=\int_{\R^d}[p_{t+h}(x,\,y)-p_t(x,\,y)]u_0(y)\,\,dy\\
&+\lambda\int_0^t\int_{\R^d}[p_{t+h-s}(x,\,y)-p_{t-s}(x,\,y)]\sigma(u_s(y))\,F(\d y\,\d s)\\
&+\lambda\int_t^{t+h}\int_{\R^d}p_{t+h-s}(x,\,y)\sigma(u_s(y))\,F(\d y\,\d s).
\end{aligned}
\end{equation*}
Since $(\sG u)_t(x)$ is in fact smooth for $t>0$, we  will look at higher moments of the remaining terms. Recall that $\sup_{x\in \R^d}\E|u_t(x)|^p$ is finite for all $t > 0$. We therefore use Burkholder's inequality together with Proposition \ref{increment} to write
\begin{equation*}
\begin{aligned}
\E \big|\int_0^t\int_{\R^d}[p_{t+h-s}(x,\,y)&-p_{t-s}(x,\,y)]\sigma(u_s(y))\,F(\d y\,\d s)\big|^p\\
&\leq c_1\big|\int_0^t\int_{\R^d} |\hat{p}_{t-s+h}(x-\xi)-\hat{p}_{t-s}(x-\xi)|^2\frac{1}{|\xi|^{d-\beta}}\,\d \xi \,\d s\big|^{p/2}\\
&\leq c_1h^{pq}.
\end{aligned}
\end{equation*}
Similarly we have 
\begin{equation*}
\begin{aligned}
\E \big| \int_t^{t+h}\int_{\R^d}p_{t+h-s}(x,\,y)&\sigma(u_s(y))\,F(\d y\,\d s) \big|^p\\
&\leq c_2\big|\int_t^{t+h}\int_{\R^d}p_{t+h-s}(x,\,y)p_{t+h-s}(x,\,z)f(y,\,z) \d y\,\d s\big|^{p/2}\\
&\leq c_3h^{(\alpha-\beta)p/2\alpha}.
\end{aligned}
\end{equation*}
We recall that $q\leq \frac{(\alpha-\beta)}{2\alpha}$ and combine the estimates above we see that 
\begin{equation*}
\E|u_{t+h}(x)-u_t(x)|^p\leq c_2h^{pq}.
\end{equation*}
Now an application of Kolmogorov's continuity theorem as in \cite{BEM2010a} completes the proof. 
\qed

\section{An extension.}
The initial conditions we have dealt with so far are functions that are non-negative on a set of positive measure. In fact, we can also deal with more general initial conditions. The only issue to achieve this extension is the existence and uniqueness of the random field solution. We will use a result of \cite{CoJoKhoShiu}, where this issue was settled whenever $u_0$ is any finite initial measure and the driving noise is white.  We have the following theorem which gives lower bounds only. The upper bound follows easily from the methods used in the previous parts of the paper. We will only briefly sketch the proof of the theorem.

\begin{theorem}
Let $u_t$ denote the unique solution to the following stochastic heat equation,
\begin{equation*}
\frac{\partial u_t(x)}{\partial t} = -(-\Delta)^{\alpha/2} u_t(x) + \lambda \sigma (u_t(x)) \dot w(t,x) \quad \text{for} ~ x \in \R ~\text{and} ~ t > 0, 
\end{equation*}
where the initial condition $u_0$ is a finite measure with $\int_{K}u_0(x)\,\d x>0$ with $K\subset \R$.  All other conditions are as described in the introduction.  Then, there exists a $T>0$ such that for $t>T$, we have 
\begin{equation*}
\inf_{x \in B(0, t^{1/ \alpha})} \E |u_t(x)|^2 \geq c \exp \left( c' \lambda^{2 \alpha/(\alpha - 1)}t \right), 
\end{equation*}
where $c$ and $c'$ are some positive constants. This immediately yields
\begin{equation*}
\liminf_{t\rightarrow \infty}\frac{1}{t}\log \E |u_t(x)|^2 \geq c' \lambda^{2 \alpha/(\alpha - 1)},
\end{equation*}
for any fixed $x\in \R$.
\end{theorem}
\begin{proof}
Recall that by Walsh's isometry, we have 
\begin{equation*}
\E|u_t(x)|^2=|(\sG u)_t(x)|^2+\lambda^2\int_0^t\int_\R p^2_{t-s}(x-y)\E|\sigma(u_s(y))|^2\,\d y \, \d s.
\end{equation*} 
See \cite{CoJoKhoShiu} for a justification of the preceding inequality.  As before we need to find a suitable lower bound on the first term. But with the current assumption on the initial condition $u_0$, the proof of Proposition \ref{estimatePDEpart} goes through and we have $(\sG u)_t(x)\geq \frac{c_1}{t^{1/\alpha}}$. We can now use the same argument as in the previous part of the paper to prove our result.
\end{proof}

One can also easily adapt the proofs in \cite{CoJoKhoShiu} to show the existence and uniqueness of the coloured noise driven equation \eqref{main-eq} when the initial condition is a finite measure. Hence all our results should hold in this case as well. 

\bibliography{Foon}
\end{document}